\documentclass{amsart}
\usepackage{amsfonts,amssymb,amsmath,amsthm,mathrsfs}
\usepackage{url}
\usepackage{enumerate}

\newtheorem{theorem}{Theorem}[section]
\newtheorem{lemma}[theorem]{Lemma}

\theoremstyle{definition}
\newtheorem{definition}[theorem]{Definition}
\newtheorem{remark}[theorem]{Remark}
\numberwithin{equation}{section}

\author[J. Chen]{Jiecheng Chen}
\address{
Jiecheng Chen, Department of  Mathematics,
Zhejiang Normal University,
Jinhua  321004,
P. R. China
}\email{jcchen@zjnu.edu.cn}
\author[G. Hu]{Guoen Hu}
\address{Guoen Hu, Department of  Applied Mathematics, Zhengzhou Information Science and Technology Institute,
Zhengzhou 450001,
P. R. China}
\email{guoenxx@163.com}
\thanks{The research of the first author was supported by the NNSF of
China under grant $\#$11671363, and the research of the second (corresponding) author was supported by
the NNSF of
China under grant $\#$11371370.}



\begin{document}

\title[Calder\'on commutator]{Weighted estimates for the Calder\'on commutator}

\begin{abstract}
In this paper, the authors consider the weighted estimates for the Calder\'on commutator defined by
\begin{eqnarray*}
&&\mathcal{C}_{m+1,\,A}(a_1,\dots,a_{m};f)(x)\\
&&\quad={\rm p.\,v.}\,\int_{\mathbb{R}}\frac{P_2(A;\,x,\,y)\prod_{j=1}^m(A_j(x)-A_j(y))}{(x-y)^{m+2}}f(y){\rm d}y,
\end{eqnarray*}
with $P_2(A;\,x,\,y)=A(x)-A(y)-A'(y)(x-y)$.
Dominating this operator by multi(sub)linear sparse operators, the authors establish the weighted bounds from
$L^{p_1}(\mathbb{R},w_1)$ $\times\dots\times L^{p_m}(\mathbb{R},w_m)$ to $L^{p}(\mathbb{R},\nu_{\vec{w}})$, with $p_1,\dots,p_m \in (1,\,\infty)$, $1/p=1/p_1+\dots+1/p_m$, and $\vec{w}=(w_1,\,\dots,\,w_m)\in A_{\vec{P}}(\mathbb{R}^{m+1})$. The authors also obtain the weighted weak type endpoint estimates for $\mathcal{C}_{m+1,\,A}$.\end{abstract}

\maketitle
\section{Introduction}

As it is well known, the   Calder\'on commutator was arisen in the study of the $L^2(\mathbb{R})$ boundedness for the Cauchy integral along Lipschitz curves. Let $A_1,\,\dots, A_m$ be functions defined on $\mathbb{R}$ such that $a_j=A_j'\in L^{q_j}(\mathbb{R})$. Define the $(m+1)$-th order  commutator of Calder\'on by\begin{eqnarray}\mathcal{C}_{m+1}(a_1,\dots,a_m;\,f)(x)=\int_{\mathbb{R}}\frac{\prod_{j=1}^m(A_j(x)-A_j(y))}{(x-y)^{m+1}}f(y){\rm d}y.\end{eqnarray}
By $T1$ Theorem and the Calder\'on-Zygmund theory, we know that  for all $p\in (1,\,\infty)$,
$$\|\mathcal{C}_{m+1}(a_1,\dots,a_m;\,f)\|_{L^p(\mathbb{R})}\lesssim \prod_{j=1}^m\|a_j\|_{L^{\infty}(\mathbb{R})}\|f\|_{L^p(\mathbb{R})},$$
and, $\mathcal{C}_{m+1}$ is bounded from $L^{\infty}(\mathbb{R})\times \dots\times L^{\infty}(\mathbb{R})\times L^1(\mathbb{R})$ to $L^{1,\,\infty}(\mathbb{R})$. For the case of $m=1$, it is known that $\mathcal{C}_{2}$ is bounded from $L^p(\mathbb{R})\times L^q(\mathbb{R})$ to $L^r(\mathbb{R})$ provided that $p,\,q\in (1,\,\infty)$ and $r\in (1/2,\,\infty)$ with $1/r=1/p+1/q$; moreover,
it is bounded from $L^p(\mathbb{R})\times L^q(\mathbb{R})$ to $L^{r,\,\infty}(\mathbb{R})$ if $\min\{p,\,q\}=1$, see \cite{ca1, ca2} for details.
By establishing the weak type endpoint estimates for multilinear singular integral operator with nonsmooth kernels, and reducing the operator $\mathcal{C}_{m+1}$ to suitable multilinear singular integral with nonsmooth kernel, Duong, Grafakos and Yan \cite{dgy} proved the following theorem.
\begin{theorem}\label{t1.1} Let $m\in\mathbb{N}$, $p_1,\,\dots,\,p_{m+1}\in [1,\,\infty)$ and $p\in (1/(m+1),\,\infty)$ with $1/p=1/p_1+\dots+1/p_{m+1}$. Then
$$\|\mathcal{C}_{m+1}(a_1,\dots,a_m;\,f)\|_{L^{p,\,\infty}(\mathbb{R})}\lesssim\prod_{j=1}^m\|a_j\|_{L^{p_j}(\mathbb{R})}
\|f\|_{L^{p_{m+1}}(\mathbb{R})}.$$
Moreover, if $\min_{1\leq j\leq m}p_j>1$, then
$$\|\mathcal{C}_{m+1}(a_1,\dots,a_m;\,f)\|_{L^{p}(\mathbb{R})}\lesssim\prod_{j=1}^m\|a_j\|_{L^{p_j}(\mathbb{R})}
\|f\|_{L^{p_{m+1}}(\mathbb{R})}.$$
\end{theorem}

Considerable attention has also been paid to the weighted estimates for $\mathcal{C}_{m+1}$. Duong, Gong, Grafakos, Li and Yan \cite{dggy} considered the weighted estimates with $A_p(\mathbb{R})$ weights for $\mathcal{C}_{m+1}$, they proved that if $p_1,\,\dots,\,p_{m+1}\in (1,\,\infty)$, $p\in (1/m,\,\infty)$ with $1/p=1/p_1+\dots+1/p_{m+1}$, then for $w\in A_p(\mathbb{R})$, $\mathcal{C}_{m+1}$ is bounded from $L^{p_1}(\mathbb{R},\,w)\times\dots\times L^{p_{m+1}}(\mathbb{R},\,w)$ to $L^p(\mathbb{R},\,w)$, here and in the following, $A_p(\mathbb{R}^n)$ denotes the weight function class of Muckenhoupt, see \cite{gra} for definitions and properties of $A_p(\mathbb{R}^n)$. Grafakos, Liu and Yang \cite{gly} considered the weighted estimates with following multiple $A_{\vec{P}}$ weights, introduced by Lerner, Ombrossi,  P\'erez,  Torres and Trojillo-Gonzalez \cite{ler}.
\begin{definition}\label{d1.1}
Let $m\in\mathbb{N}$, $w_1,\dots,w_m$ be
weights, $p_1,\dots,p_m\in[1,\,\infty)$, $p\in [1/m,\,\infty)$ with
$1/p=1/p_1+\dots+1/p_m$. Set $\vec{w}=(w_1,\,\dots,\,w_m)$,
$\vec{P}=(p_1,\,...,\,p_m)$ and
$\nu_{\vec{w}}=\prod_{k=1}^{m}w_k^{p/p_k}$. We say that $\vec{w}\in
A_{\vec{P}}(\mathbb{R}^{mn})$ if the $A_{\vec{P}}(\mathbb{R}^{mn})$ constant of $\vec{w}$, defined by
$$[\vec{w}]_{A_{\vec{P}}}=\sup_{Q\subset
\mathbb{R}^n}\Big(\frac{1}{|Q|}\int_Q\nu_{\vec{w}}(x)\,{\rm
d}x\Big)
\prod_{k=1}^m\Big(\frac{1}{|Q|}\int_Qw_k^{-\frac{1}{{p_k}-1}}(x)\,{\rm
d}x\Big)^{p/p_k'},$$ is finite, here and in the following, for $r\in [1,\,\infty)$, $r'=\frac{r}{r-1}$; when
$p_k=1$,
$\Big(\frac{1}{|Q|}\int_Qw_k^{-\frac{1}{p_k-1}}\Big)^{\frac{1}{p_k'}}$ is understood as $(\inf_{Q}w_k\big)^{-1}$.\end{definition}

Using some new maximal operators, Grafakos, Liu and Yang \cite{gly} proved that if $p_1,\,\dots,\,p_{m+1}\in [1,\,\infty)$ and $p\in [\frac{1}{m+1},\,\infty)$ with $1/p=1/p_1+\dots+1/p_{m+1}$,  and $\vec{w}=(w_1,\,\dots,\,w_m, w_{m+1})\in A_{\vec{P}}(\mathbb{R}^{m+1})$, then $\mathcal{C}_{m+1}$ is bounded
from $L^{p_1}(\mathbb{R},\,w_1)\times \dots\times L^{p_{m+1}}(\mathbb{R},\,w_m)$ to $L^{p,\,\infty}(\mathbb{R}, \nu_{\vec{w}})$, and when $\min_{1\leq j\le m+1}p_j>1$, $\mathcal{C}_{m+1}$ is bounded from $L^{p_1}(\mathbb{R}, w_1)\times \dots\times L^{p_{m+1}}(\mathbb{R},w_{m+1})$ to $L^{p}(\mathbb{R},\,\nu_{\vec{w}})$.
Fairly recently, by dominating multilinear singular integral operators by sparse operators, Chen and Hu \cite{chenhu} improved the result of Grafakos et al. in \cite{gly}, and obtain the following quantitative weighted bounds for $\mathcal{C}_{m+1}$.
\begin{theorem}Let $m\in\mathbb{N}$, $p_1,\,\dots,\,p_{m+1}\in (1,\,\infty)$ and $p\in (1/(m+1),\,\infty)$ with $1/p=1/p_1+\dots+1/p_{m+1}$, $\vec{w}=(w_1,\,\dots,\,w_{m+1})\in A_{\vec{P}}(\mathbb{R}^{m+1})$. Then
\begin{eqnarray}\label{eq1.2}&&\|\mathcal{C}_{m+1}(a_1,\dots,a_m;\,f)\|_{L^{p}(\mathbb{R},\,\nu_{\vec{w}})}\\
&&\qquad\lesssim[\vec w]_{A_{\vec P}}^{\max \{1, \frac{p_1'}{p},\cdots, \frac{p'_{m+1}}{p}\}}\prod_{j=1}^m\|a_j\|_{L^{p_j}
(\mathbb{R},\,w_j)}\|f\|_{L^{p_{m+1}}(\mathbb{R},w_{m+1})}.\nonumber
\end{eqnarray}
 \end{theorem}
We remark that the quantitative weighted bounds for classical operators in harmonic analysis was begun by Buckley \cite{bu} and then by many other authors, see \cite{hyt,pet2,hlp,hp2, ler2,ler3, lern, lms} and references therein.

Observe that (\ref{eq1.2}) also hold if $\max_{1\leq j\leq m}p_j=\infty$ but $p\in (\frac{1}{m+1},\,\infty)$ (in this case,
$\|a_j\|_{L^{\infty}
(\mathbb{R},\,w_j)}$ should be replaced by $\|a_j\|_{L^{\infty}
(\mathbb{R})}$ and $w_k$ should be replaced by $1$ if $p_k=\infty$). A natural question is: if a result similar to (\ref{eq1.2}) holds true when $a_j\in {\rm BMO}(\mathbb{R})$ for some $1\leq j\leq m$? In this paper, we consider the operator defined by
\begin{eqnarray}\label{eq1.3}
&&\mathcal{C}_{m+1,\,A}(a_1,\dots,a_{m};\,f)(x)\\
&&={\rm p.\,v.}\,\int_{\mathbb{R}}\frac{P_2(A;\,x,\,y)\prod_{j=1}^m(A_j(x)-A_j(y))}{(x-y)^{m+2}}f(y){\rm d}y,\nonumber
\end{eqnarray}
with $P_2(A;\,x,\,y)=A(x)-A(y)-A'(y)(x-y)$. When $a_1,\dots,a_m\in L^{\infty}(\mathbb{R})$,  it is obvious that $\prod_{j=1}^m(A_j(x)-A_j(y))(x-y)^{-m-1}$ is a Calder\'on-Zygmund kernel. Repeating the argument in \cite{cohen}, we know that for any $p\in (1,\,\infty)$,
\begin{eqnarray}\label{equa:1.4}\|\mathcal{C}_{m+1,\,A}(a_1,\dots,a_{m};f)\|_{L^p(\mathbb{R})}\lesssim \|A'\|_{{\rm BMO}(\mathbb{R})}\prod_{j=1}^{m}\|a_j\|_{L^{\infty}(\mathbb{R})}\|f\|_{L^p(\mathbb{R})}.\end{eqnarray}
Moreover,  the results in \cite{hy} implies that
for each $\lambda>0$,
$$|\{x\in\mathbb{R}: \mathcal{C}_{m+1,\,A}(a_1,\dots,a_{m};f)(x)>\lambda\}|\lesssim_{a_1,\dots,a_m} \int_{\mathbb{R}}\frac{|f(x)|}{\lambda}\log \Big({\rm e}+\frac{|f(x)|}{\lambda}\Big){\rm d}x.$$
Operators like $\mathcal{C}_{m+1,\,A}$  with $a_j\in L^{\infty}(\mathbb{R})$ were introduced by Cohen \cite{cohen}, and then considered by Hofmann \cite{hof} and other authors, see also \cite{hu,huli, hy} and the related references therein.

Our main purpose in this paper is to establish the weighted bound similar to (1.2) for the operator $\mathcal{C}_{m+1,\,A}$ in (\ref{eq1.3}). For a weight $u\in A_{\infty}(\mathbb{R}^n)=\cup_{p\geq 1}A_p(\mathbb{R}^n)$, $[u]_{A_{\infty}}$, the $A_{\infty}$ constant of $u$, is defined by
$$[u]_{A_{\infty}}=\sup_{Q\subset \mathbb{R}^n}\frac{1}{u(Q)}\int_{Q}M(u\chi_Q)(x){\rm d}x.$$
Recall that  for $p_1,\,\dots,\,p_m\in [1,\,\infty)$,  $\vec{w}=(w_1,\,\dots,\,w_m)\in A_{\vec{P}}(\mathbb{R}^{mn})$ if and only if
$\nu_{\vec{w}}\in A_{mp}(\mathbb{R}^n)$ and $w_j^{-\frac{1}{p_j-1}}\in A_{mp_j'}(\mathbb{R}^n)$ (see \cite{ler} for details). Our main result  can be stated as follows.
\begin{theorem}\label{t1.3}
Let $m\in\mathbb{N}$, $p_1,\dots,p_{m+1}\in [1,\infty)$, $p\in [\frac{1}{m+1},\infty)$ with $1/p=1/p_1+\dots+1/p_{m+1}$,  $\vec{w}=(w_1,\dots,w_{m+1})\in A_{\vec{P}}(\mathbb{R}^{m+1})$,  $A'\in {\rm BMO}(\mathbb{R})$ with $\|A'\|_{{\rm BMO}(\mathbb{R})}=1$.  \begin{itemize}
\item[\rm (i)] If $\min_{1\leq j\leq m+1}p_j>1$, then
\begin{eqnarray*}\|\mathcal{C}_{m+1,A}(a_1,\dots,a_m;f)\|_{L^p(\mathbb{R},\nu_{\vec{w}})}&\lesssim&[\vec w]_{A_{\vec P}}^{\max \{1,\frac{p_1'}{p},\cdots, \frac{p'_{m+1}}{p}\}}\big[w_m^{-\frac{1}{p_m-1}}\big]_{A_{\infty}}\\
&&\times
\|f\|_{L^{p_{m+1}}(\mathbb{R},w_{m+1})}\prod_{j=1}^m\|a_j\|_{L^{p_j}(\mathbb{R},w_j)};\end{eqnarray*}
\item[\rm (ii)] if $p_1=\dots=p_{m+1}=1$, then for each $\lambda>0$,
\begin{eqnarray*}
&&\nu_{\vec{w}}(\{x\in\mathbb{R}:\,|\mathcal{C}_{m+1,A}(a_1,\dots,a_m;\,f)(x)|>\lambda\}\big)\\
&&\quad\lesssim \Big(\prod_{j=1}^m\int_{\mathbb{R}}\frac{|a_j(y_j)|}{\lambda^{\frac{1}{m+1}}}\log\Big({\rm e}+\frac{|a_j(y_j)|}{\lambda^{\frac{1}{m+1}}}\Big)w_j(y_j)w_j(y_j){\rm d}y_j\Big)^{\frac{1}{m+1}}\\
&&\quad\times\Big(\int_{\mathbb{R}}\frac{|f(y)|}{\lambda^{\frac{1}{m+1}}}\log\Big({\rm e}+\frac{|f(y)|}{\lambda^{\frac{1}{m+1}}}\Big)w_{m+1}(y){\rm d}y\Big)^{\frac{1}{m+1}}.
\end{eqnarray*}
\end{itemize}
\end{theorem}
\begin{remark}To prove Theorem \ref{t1.3}, we will employ a suitable variant of the ideas of Lerner \cite{ler2} (see also \cite{chenhu,li} in the case of multilinear operator),   to   dominate  $C_{m+1,\,A}$ by multilinear sparse operators. This argument need certain weak type endpoint estimates for  the grand maximal operator of $\mathcal{C}_{m+1,A}$. Although $K_A(x;y_1,\dots,y_{m+1})$, the kernel of the multilinear singular integral operator  $\mathcal{C}_{m+1,A}$, enjoys the nonsmooth kernel conditions about the variable $y_1,\,\dots,\,y_m$ as in \cite{dgy},  we do not know $K_A(x;\,y_1,\,\dots,\,y_{m+1})$ enjoys any condition about the variable $y_{m+1}$. Our argument is a modification of the proof of Theorem 1.1 in \cite{dgy},  based on a local estimate (see Lemma \ref{lem2.4} below), and involves
the combination of sharp function estimates and the argument used in \cite{dgy}.
\end{remark}

In what follows, $C$ always denotes a
positive constant that is independent of the main parameters
involved but whose value may differ from line to line. We use the
symbol $A\lesssim B$ to denote that there exists a positive constant
$C$ such that $A\le CB$. Specially, we use $A\lesssim_{p} B$ to denote that there exists a positive constant
$C$ depending only on $p$ such that $A\le CB$. Constant with subscript such as $C_1$,
does not change in different occurrences. For any set $E\subset\mathbb{R}^n$,
$\chi_E$ denotes its characteristic function.  For a cube
$Q\subset\mathbb{R}^n$ (interval $I\subset \mathbb{R}$) and $\lambda\in(0,\,\infty)$, we use
$\lambda Q$ to denote the cube with the same center as $Q$ and whose
side length is $\lambda$ times that of $Q$. For $x\in\mathbb{R}^n$ and $r>0$, $B(x,\,r)$ denotes the ball centered at $x$ and having radius $r$.

\section{An endpoint estimate}
This section is devoted to an endpoint estimate for $\mathcal{C}_{m+1,A}$. We begin with a preliminary lemma.
\begin{lemma}\label{lem2.1}
Let $A$ be a function on $\mathbb{R}^n$ with derivatives of order one in $L^q(\mathbb{R}^n)$ for some $q\in (n,\,\infty]$. Then
$$|A(x)-A(y)| \lesssim|x-y|\Big(\frac{1}{|I_x^y|}\int_{I_x^y}|\nabla A(z)|^q{\rm d}z\Big)^{\frac{1}{q}},$$
where
$I_x^y$ is the cube centered at $x$ and having side length $2|x-y|.$
\end{lemma}
For the proof of Lemma \ref{lem2.1}, see \cite{cohen}.

For $\gamma\in [0,\,\infty)$ and a cube $Q\subset \mathbb{R}^n$, let $\|\cdot\|_{L(\log L)^{\gamma},\,Q}$ be the Luxmberg norm defined by$$\|f\|_{L(\log L)^{\gamma},\,Q}=\inf\Big\{\lambda>0:\,\frac{1}{|Q|}\int_{Q}\frac{|f(y)|}{\lambda}\log^{\gamma}\Big({\rm e}+\frac{|f(y)|}{\lambda}\Big){\rm d}y\leq 1\Big\}.$$
Define the maximal operator $M_{L(\log L)^{\gamma}}$ by
$$M_{L(\log L)^{\gamma}}f(x)=\sup_{Q\ni x}\|f\|_{L(\log L)^{\gamma},\,Q}.$$
Obviously, $M_{L(\log L)^{0}}$ is just the Hardy-Littlewood maximal operator $M$. It is well known that $M_{L(\log L)^{\gamma}}$ is bounded on $L^p(\mathbb{R}^n)$ for all $p\in (1,\,\infty)$, and for $\lambda>0$,
\begin{eqnarray}\label{eq2.01}&&|\{x\in\mathbb{R}^n:\,M_{L(\log L)^{\gamma}}f(x)>\lambda\}|\lesssim \int_{\mathbb{R}^n}\frac{|f(x)|}{\lambda}\log^{\gamma}\Big({\rm e}
+\frac{|f(x)|}{\lambda}\Big){\rm d}x.\end{eqnarray}

Let $s\in (0,\,1/2)$  and $M^{\sharp}_{0,\,s}$ be the John-Str\"omberg sharp maximal operator defined by
$$M^{\sharp}_{0,\,s}f(x)=\sup_{Q\ni x}\inf_{c\in\mathbb{C}}\inf\big\{t>0:\,|\{y\in Q:\,|f(y)-c|>t\}|<s|Q|\big\},$$
where the supremum is taken over all cube containing $x$. This operator was introduced by John \cite{john} and recovered by Str\"omberg in \cite{stro}.

\begin{lemma}\label{lem2.2}
Let  $\Phi$ be a  increasing function on $[0,\,\infty)$ which satisfies the doubling condition that
$$\Phi(2t)\leq C\Phi(t),\,t\in [0,\,\infty).$$
Then there exists a constant $s_0\in (0,\,1/2)$, such that for any $s\in (0,\,s_0]$,
$$\sup_{\lambda>0}\Phi(\lambda)\big|\{x\in\mathbb{R}^n:\,|h(x)|>\lambda\}\big|\lesssim \sup_{\lambda>0}\Phi(\lambda)\big|\{x\in\mathbb{R}^n:\,M^{\sharp}_{0,\,s}h(x)>\lambda\}\big|,$$
provided that
$$\sup_{\lambda>0}\Phi(\lambda)\big|\{x\in\mathbb{R}^n:\,|h(x)|>\lambda\}\big|<\infty.$$
\end{lemma}
This lemma can be proved by repeating the proof of Theorem 2.1 in \cite{hyy}. We omit the details for brevity.
\begin{lemma}\label{lem2.extra}
Let $R>1$. There exists a constant $C(n,\,R)$ such that for all open set $\Omega\subset \mathbb{R}^n$,  $\Omega$ can be decomposed as
$\Omega=\cup_{j}Q_j$, where $\{Q_j\}$ is a sequence of cubes with disjoint interiors, and
\begin{itemize}
\item[\rm (i)]
$$5R\le \frac{{\rm dist}(Q_j,\,\mathbb{R}^n\backslash \Omega)}{{\rm diam} Q_j}\le 15R,$$
\item[\rm (ii)] $\sum_{j}\chi_{RQ_j}(x)\le C_{n,\,R} \chi_{\Omega}(x).$
\end{itemize}
\end{lemma}
For the proof of Lemma \ref{lem2.extra}, see \cite[p.\,256]{saw}.

We return to $\mathcal{C}_{m+1}$. As it was proved in \cite{dgy}, $\mathcal{C}_{m+1}$ can be rewritten as the following multilinear singular integral operator
\begin{eqnarray*}&&\mathcal{C}_{m+1}(a_1,\,\dots,\,a_m;\,f)(x)\\
&&\quad=\int_{\mathbb{R}^{m+1}}K(x; y_1,\dots, y_{m+1})\prod_{j=1}^ma_j(y_j)f(y_{m+1}) {\rm d}y_1 \dots {\rm d}y_{m+1};\nonumber
\end{eqnarray*}
where
\begin{eqnarray}\label{eq2.02}
&&K(x;y_1,\dots, y_{m+1}) =
\frac{(-1)^{me(y_{m+1}-x)}}{(x-y_{m+1})^{m+1}}
\prod_{j=1}^m
\chi_{(x \wedge y_{m+1},\, x\vee  y_{m+1})}(y_j),
\end{eqnarray}
$e$ is the characteristic function of $[0,\,\infty)$, $x\wedge  y_{m+1}=\min\{x,\,y_{m+1}\}$ and $x\vee y_{m+1}=\max\{x,\,y_{m+1}\}$. Obviously, for $x,\,y_1,\,\dots,\,y_{m+1}\in \mathbb{R}$,
\begin{eqnarray}\label{eq2.03}
|K(x;y_1,\dots, y_{m+1})|\lesssim
\frac{1}{(\sum_{j=1}^{m+1}|x-y_j|)^{m+1}}.\end{eqnarray}

\begin{lemma}\label{lem2.3}
Let ${K}$ be the same as in (\ref{eq2.02}). Then for $x,\,x',\,y_1,\,\dots,\,y_{m+1}\in \mathbb{R}$ with $12|x-x'|<\min_{1\leq j\leq m+1}|x-y_j|$
$$|K(x; y_1,\,\dots,\, y_{m+1})-K(
x';\,y_1,\,\dots,\,y_{m+1})|\\
\lesssim\frac{|x-x'|}{
\big(\sum_{j=1}^{m+1}|x-y_j|)^{m+2}}.$$
\end{lemma}
For the proof of Lemma \ref{lem2.3}, see \cite{huz}.

\begin{lemma}\label{lem2.4}Let $A$ be a function on $\mathbb{R}$ such that $A'\in {\rm BMO}(\mathbb{R})$, $a_1,\,\dots,\,a_m\in L^{1}(\mathbb{R})$. Then for $\tau\in (0,\,\frac{1}{m+2})$ and any interval $I\subset \mathbb{R}$,
\begin{eqnarray}\label{eq2.04}&&\Big(\frac{1}{|I|}\int_I\big|\mathcal{C}_{m+1,\,A}(a_1,\dots,a_m;f\chi_I)(y)\big|^{\tau}{\rm d}y\Big)^{\frac{1}{\tau}}\lesssim\|f\|_{L\log L,\,4I}\prod_{j=1}^{m}
\langle |a_j|\rangle_{4I}.\end{eqnarray}
\end{lemma}
\begin{proof} For a fixed interval $I\subset \mathbb{R}$, let $\varphi\in C^{\infty}_0(\mathbb{R})$ such that $0\leq \varphi(y)\leq 1$, $\varphi(y)\equiv 1$ for $y\in I$, ${\rm supp}\,\varphi\subset 2I$ and $\|\varphi'\|_{L^{\infty}(\mathbb{R})}\lesssim |I|^{-1}$. Set $$A_I(y)=A(y)-\langle A'\rangle_Iy,\,\,\,A^{\varphi}(y)=(A_I(y)-A_I(y_0))\varphi(y)$$ with $y_0\in 3I\backslash 2I$, and let $a^{\varphi}(y)=(A^{\varphi})'(y)$. Applying Lemma \ref{lem2.1}, we know that
$$|A_{I}(y)-A_I(y_0)|\lesssim |I|.$$
Thus for $y\in I$,
\begin{eqnarray*}|a^{\varphi}(y)|&\lesssim&\Big(\frac{1}{|I|}|A_I(y)-A_I(y_0)|+|A'(y)-\langle A'\rangle_I|\Big)\chi_{2I}(y)\\
&\lesssim&\big(1+|A'(y)-\langle A'\rangle_I|\big)\chi_{2I}(y).
\end{eqnarray*}
This, in turn implies that
$$\|a^{\varphi}\|_{L^1(\mathbb{R})}\lesssim \|A'\|_{{\rm BMO}(\mathbb{R})}|I|,$$
and by the generalization of H\"older inequality (see \cite[p. 64]{rr}),
$$\|a^{\varphi}f\chi_I\|_{L^1(\mathbb{R})}\lesssim |I|\|f\|_{L\log L,I}.$$
For $j=1,\,\dots,\,m$, let $A_j^{\varphi}(z)=\big(A_j(z)-A_j(y_0)\big)\varphi(z),$ and $a_j^{\varphi}(z)=(A_j^{\varphi})'(z)$.
It then follows that
$$\|a^{\varphi}_j\|_{L^1(\mathbb{R})}\lesssim\int_{4I}|a_j(z)|{\rm d}z.$$
For $y\in I$, write
\begin{eqnarray*}
&&\mathcal{C}_{m+1,\,A}(a_1,\dots,a_m;\,f\chi_I)(y)\\
&&\quad=\int_{\mathbb{R}}\frac{\prod_{j=1}^m(A_j^{\varphi}(y)-A_j^{\varphi}(z))(A^{\varphi}(y)-A^{\varphi}(z))}{(y-z)^{m+2}}f(z)\chi_{I}(z){\rm d}z\\
&&\qquad+\int_{\mathbb{R}}\frac{\prod_{j=1}^m(A_j^{\varphi}(y)-A_j^{\varphi}(z))}{(y-z)^{m+1}}a^{\varphi}(z)f(z)\chi_I(z){\rm d}z\\
&&\quad=\mathcal{C}_{m+2}(a_1^{\varphi},\,\dots,a_m^{\varphi},a^{\varphi};\,f\chi_I)(y)+\mathcal{C}_{m+1}\big(a_1^{\varphi},\dots,a_m^{\varphi},
a^{\varphi}f\chi_I\big)(y).
\end{eqnarray*}
Theorem \ref{t1.1} tells us that $\mathcal{C}_{m+2}(a_1^{\varphi},\,\dots,a_m^{\varphi},a^{\varphi};\,f\chi_I)$ is bounded from
$L^1(\mathbb{R})\times\dots\times L^1(\mathbb{R})$ to $L^{\frac{1}{m+2},\,\infty}(\mathbb{R})$. As in the proof of Kolmogorov's inequality, we can deduce that  $\tau\in (0,\,\frac{1}{m+2})$,
\begin{eqnarray*}
&&\Big(\frac{1}{|I|}\int_{I}\Big|\mathcal{C}_{m+2}(a_1^{\varphi},\dots,a_m^{\varphi},a^{\varphi};f\chi_I)(y)\Big|^{\tau}{\rm d}y\Big)^{\frac{1}{\tau}}\\
&&\quad\lesssim |I|^{-m-2}\prod_{j=1}^m\|a_j^{\varphi}\|_{L^1(\mathbb{R})}\|f\chi_I\|_{L^1(\mathbb{R})}\|a^\varphi\|_{L^1(\mathbb{R})}\lesssim
\langle|f|\rangle_I\prod_{j=1}^m\langle|a_j|\rangle_{4I},
\end{eqnarray*} On the other hand, since $\mathcal{C}_{m+1}$ is bounded from
$L^1(\mathbb{R})\times \dots\times L^1(\mathbb{R})$ to $L^{\frac{1}{m+1},\,\infty}(\mathbb{R})$, we then know that for $\varsigma\in (0,\,\frac{1}{m+1})$,
\begin{eqnarray*}
&&\Big(\frac{1}{|I|}\int_{I}\Big|\mathcal{C}_{m+1}\big(a_1^{\varphi},\dots,a_m^{\varphi};a^{\varphi}f\chi_I\big)(y)\Big|^{\varsigma}{\rm d}y\Big)
^{\frac{1}{\varsigma}}\\
&&\quad\lesssim |I|^{-m-1}\prod_{j=1}^m\|a_j^{\varphi}\|_{L^1(\mathbb{R})}\|a^{\varphi}f\chi_I\|_{L^1(\mathbb{R})}\lesssim \|f\|_{L\log L,\,I}\prod_{j=1}^m\langle|a_j|\rangle_{4I},
\end{eqnarray*}
Combining the last two inequality yields (\ref{eq2.04}).
\end{proof}

Now we rewrite $\mathcal{C}_{m+1,\,A}$
as the following multilinear singular integral operator,
\begin{eqnarray*}
\mathcal{C}_{m+1,\,A}(a_1,\dots,a_{m};f)(x)=\int_{\mathbb{R}^{m+1}}K_A(x;y_1,\dots,y_{m+1})\prod_{j=1}^ma_j(y_j)f(y_{m+1}){\rm d}\vec{y},\end{eqnarray*}
where and in the following,
\begin{eqnarray}\label{eq2.05}K_A(x;y_1,\dots,y_{m+1})=K(x;y_1,\dots,y_{m+1})\frac{P_2(A;\,x,\,y_{m+1})}{(x-y_{m+1})},
\end{eqnarray}with  $K(x;y_1,\dots,y_{m+1})$  defined by (\ref{eq2.02}). Obviously,
\begin{eqnarray}\label{eq2.06}|K_A(x;y_1,\dots,y_{m+1})|\lesssim\frac{1}{(\sum_{j=1}^{m+1}|x-y_j|)^{m+2}}|P_2(A;\,x,\,y_{m+1})|.
\end{eqnarray}
\begin{lemma}\label{lem2.5}
Let $\phi\in C^{\infty}(\mathbb{R})$ be even, $0\leq \phi\leq 1$, $\phi(0)=0$ and ${\rm supp}\,\phi\subset [-1,\,1]$. Set $\Phi(t)=\phi'(t)$, $\Phi_t(y)=t^{-1}\Phi(x/t)$ and $k_t(x,\,y)=\Phi_t(x-y)\chi_{(x,\infty)}(y).$ For $j=1,\,\dots,m$, set $$K_{A,\,t}^j(x;\,y_1,\,\dots,\,y_m) =\int_{\mathbb{R}^n}K_A(x;\,y_1,\dots,y_{j_1},z,y_{j+1},\dots,y_{m+1})k_t(z,\,y_j){\rm d}z.$$
Then for $j=1,\,\dots,m$, $x,\,y_1,\,\dots,\,y_{m+1}\in\mathbb{R}$ and  $t>0$ with $2t\leq |x-y_j|$,
\begin{eqnarray*}&&|K_A(x;y_1,\dots,y_{m+1})-K_{A,\,t}^j(x;y_1,\dots,y_{m+1})|\\
&&\quad\lesssim\frac{|P_2(A,\,x,\,y_{m+1})|}{(\sum_{k=1}^{m+1}|x-y_k|)^{m+2}}\phi\Big(\frac{|y_{m+1}-y_j|}{t}\Big).
\end{eqnarray*}
\end{lemma}
\begin{proof} We only consider $j=1$.  Write
\begin{eqnarray*}
&&K_A(x;y_1,\dots,y_{m+1})-K_{A,\,t}^1(x;y_1,\dots,y_{m+1})\\
&&=\frac{(-1)^{ me(y_{m+1}-x)}}{(x-y_{m+1})^{m+1}}\frac{P_2(A;\,x,\,y_{m+1})}{(x-y_{m+1})}\prod_{j=2}^m
\chi_{(x\wedge y_{m+1},\, x\vee y_{m+1})}(y_j)\\
&&\qquad\times\Big(\chi_{(x\wedge y_{m+1},\, x\vee y_{m+1})}(y_1)\\
&&\qquad\qquad-\int_{-\infty}^{y_1}\chi_{(x\wedge y_{m+1},\, x\vee y_{m+1})}(z)k_t(z-y){\rm d}z\Big)
\end{eqnarray*}
From the proof of Theorem 4.1 in  \cite{dgy}, we find that  when $|x-y_1|>2t$,
\begin{eqnarray*}
&&\Big|\chi_{(x\wedge y_{m+1},\, x\vee y_{m+1})}(y_1)-\int_{-\infty}^{y_1}\chi_{(x\wedge y_{m+1},\, x\vee y_{m+1})}(z)k_t(z-y){\rm d}z\Big|\\
&&\quad\lesssim \phi\Big(\frac{|y_{m+1}-y_1|}{t}\Big).
\end{eqnarray*}
Note that $|K_A(x;y_1,\dots,y_{m+1})-K_{A,\,t}^1(x;y_1,\dots,y_{m+1})|\not =0$ only if $|x-y_{m+1}|>\max_{1\leq k\leq m}|x-y_k|$.
Our desired conclusion then follows directly.
\end{proof}
\begin{remark}We do not know if  $K_A(x;y_1,\dots,y_{m+1})$ enjoys  the properties as Lemma \ref{lem2.5} about the variable $y_{m+1}$.
\end{remark}
We now recall the approximation to the indentity introduced by Douong and McIntosh \cite{duongmc}.
\begin{definition}\label{defn2.1}
A family of operators $\{D_t\}_{t>0}$ is said to be an approximation to the identity in $\mathbb{R}$, if for
every $t>0$, $D_t$ can be represented by the kernel at in the following sense: for every function $u\in L^p(\mathbb{R})$
with $p\in [1,\,\infty]$ and a. e.  $x\in\mathbb{R}$,
$$D_tu(x)=\int_{\mathbb{R}}a_t(x,\,y)u(y)dy,$$
and the kernel $a_t$ satisfies that for all $x,\,y\in\mathbb{R}$ and $t>0$,
$$|a_t(x,\,y)|\le h_t(x,\,y)=t^{-1/s}h\Big(\frac{|x-y|}{t^{1/s}}\Big),$$
where $s>0$ is a constant and $h$ is a positive, bounded and decreasing function such that for some constant $\eta>0$,
$$\lim_{r\rightarrow\infty}r^{1+\eta}h(r)=0.$$
\end{definition}
\begin{lemma}\label{lem2.7}Let $A$ be a function on $\mathbb{R}$ such that $A'\in {\rm BMO}(\mathbb{R})$, $q_1,\,\dots,\,q_{m+1}\in [1,\,\infty)$. Suppose that
for some  $\beta\in [0,\,\infty)$, $\mathcal{C}_{m+1,A}$ satisfies the estimate that
\begin{eqnarray*}&&|\{x\in\mathbb{R}^n:\,|\mathcal{C}_{m+1,A}(a_1,\dots,a_m;\,f)(x)|>1\}|\\
&&\quad\lesssim \sum_{j=1}^{m}\|a_j\|_{L^{q_j}(\mathbb{R})}^{q_j}+\int_{\mathbb{R}^n}|f(x)|^{q_{m+1}}
\log^{\beta}\big({\rm e}+|f(x)|\big){\rm d}x.
\end{eqnarray*}Then for $p_j\in [1,\,q_j)$, $j=1,\,\dots,\,m$
\begin{eqnarray*}&&|\{x\in\mathbb{R}:\,|\mathcal{C}_{m+1,A}(a_1,\dots,a_m;\,f)|>1\}|\\
&&\quad \lesssim\sum_{j=1}^{m}\|a_j\|_{L^{p_j}(\mathbb{R})}^{p_j}+
\int_{\mathbb{R}}|f(x)|^{q_{m+1}}
\log^{\beta_{q_{m+1}}}\big({\rm e}+|f(x)|\big){\rm d}x,
\end{eqnarray*}
where $\beta_{q_{m+1}}=\beta$ if $q_{m+1}\in (1,\,\infty)$ and $\beta_{q_{m+1}}=\max\{1,\,\beta\}$  if $q_{m+1}=1$.
\end{lemma}
\begin{proof} We employ the ideas in \cite{dgy}, together with some modifications. At first, we prove that
\begin{eqnarray}\label{eq2.07}&&|\{x\in\mathbb{R}:\,|\mathcal{C}_{m+1,A}(a_1,\dots,a_m;\,f)|>1\}|\\
&&\quad \lesssim\|a_1\|_{L^{p_1}(\mathbb{R})}^{p_1}+\sum_{j=2}^{m}\|a_j\|_{L^{q_j}(\mathbb{R}^n)}^{q_j}+
\int_{\mathbb{R}}|f(x)|^{q_{m+1}}
\log^{\beta}\big({\rm e}+|f(x)|\big){\rm d}x.\nonumber
\end{eqnarray}
To do this, we apply Lemma \ref{lem2.extra} to  the set $$\Omega=\{x\in\mathbb{R}:\, M(|a_1|^{p_1})(x)>1\},$$ and obtain a sequence of intervals $\{I_l\}$ with disjoint interiors, such that $$\frac{1}{|I_l|}\int_{I_l}|a_1(x)|^{p_1}{\rm d}x\lesssim 1,$$
and $\sum_{l}\chi_{4I_l}(x)\lesssim \chi_{\Omega}(x)$.
Let $D_t$ be the integral operator defined by
$$
D_th(x)=\int_{\mathbb{R}}k_t(x,\,y)h(y){\rm d}y,
$$
with $k_t$ the same as in Lemma \ref{lem2.5}. Then $\{D_t\}_{t>0}$ is an approximation to the identity in the sense of Definition \ref{defn2.1}.
Set
$$a_1^1(x)=a_1(x)\chi_{\mathbb{R}^n\backslash\Omega}(x),\,\,\,a_1^2(x)=\sum_{l}D_{|I_l|}b_{1}^{l}(x),$$
and$$a_1^3(x)=\sum_{l}\big(b_{1}^{l}(x)-D_{|I_l|}b_{1}^{l}(x)\big),$$
with $b_1^{l}(y)=a_1(y)\chi_{I_l}(y)$.
Obviously, $\|b_l^l\|_{L^{p_1}(\mathbb{R}^n)}\lesssim 1$ and $\|a_1^1\|_{L^{\infty}(\mathbb{R})}\lesssim 1$.  Our hypothesis states that
\begin{eqnarray*}
&&\big|\big\{x\in \mathbb{R}:\,\big|\mathcal{C}_{m+1,A}(a_1^1,\dots,a_m;\,f)(x)|>1\big\}\big|\\
&&\quad\lesssim
\|a_1\|_{L^{p_1}(\mathbb{R})}^{p_1}+\sum_{j=2}^{m}\|a_j\|_{L^{q_j}(\mathbb{R})}^{q_j}+
\int_{\mathbb{R}^n}|f(x)|^{q_{m+1}}
\log^{\beta}\big({\rm e}+|f(x)|\big){\rm d}x
\end{eqnarray*}
On the other hand, as it was pointed out in \cite[p. 241]{duongmc}, we know that
$$\|a_1^2\|_{L^{q_1}(\mathbb{R})}\lesssim\Big\|\sum_{l}\chi_{I_l}\Big\|_{L^{q_1}(\mathbb{R})}\lesssim \big(\sum_{l}|Q_l|\big)^{1/q_1} \lesssim \|a_1\|_{L^{p_1}(\mathbb{R})}^{p_1/q_1}.$$
Thus,
\begin{eqnarray*}
&&\big|\big\{x\in \mathbb{R}:\,\big|\mathcal{C}_{m+1,A}(a_1^2,\dots,a_m;\,f)(x)|>1\big\}\big|\\
&&\quad\lesssim
\|a_1^2\|_{L^{q_1}(\mathbb{R})}^{q_1}+\sum_{j=2}^{m}\|a_j\|_{L^{q_j}(\mathbb{R})}^{q_j}+
\int_{\mathbb{R}}|f_{m+1}(x)|^{q_{m+1}}
\log^{\beta}\big({\rm e}+|f_m(x)|\big){\rm d}x\\
&&\quad\lesssim
\|a_1\|_{L^{p_1}(\mathbb{R})}^{p_1}+\sum_{j=2}^{m}\|a_j\|_{L^{q_j}(\mathbb{R})}^{q_j}+
\int_{\mathbb{R}}|f_{m+1}(x)|^{q_{m+1}}
\log^{\beta}\big({\rm e}+|f_m(x)|\big){\rm d}x.
\end{eqnarray*}
Our proof for (\ref{eq2.07}) is now reduced to proving
\begin{eqnarray}\label{eq2.08}
&&\big|\big\{x\in \mathbb{R}:\,\big|\mathcal{C}_{m+1,A}\big(a_1^{3},\,\dots,\,a_{m},f\big)(x)|>1\big\}\big|\lesssim
\|a_1\|_{L^{p_1}(\mathbb{R})}^{p_1}\\
&&\quad+\sum_{j=2}^{m}\|a_j\|_{L^{q_j}(\mathbb{R})}^{q_j}+
\int_{\mathbb{R}^n}|f_{m+1}(x)|^{q_{m+1}}
\log^{\tilde{\beta}_{q_{m+1}}}\big({\rm e}+|f_{m+1}(x)|\big){\rm d}x,\nonumber
\end{eqnarray}
here, $\tilde{\beta}_{q_{m+1}}=0$ if $q_{m+1}\in (1,\,\infty)$ and $\tilde{\beta}_{q_{m+1}}=1$ if $q_{m+1}=1$.

We now prove (\ref{eq2.08}). Let $\widetilde{\Omega}=\cup_{l}16I_l$. It is obvious that $$|\widetilde{\Omega}|\lesssim \|a_1\|_{L^{p_1}(\mathbb{R})}^{p_1}.$$
For each $x\in\mathbb{R}\backslash\widetilde{\Omega}$, by Lemma \ref{lem2.5}, we can write
\begin{eqnarray*}
&&\big|\mathcal{C}_{m+1,A}(a_1^3,\,a_2,\dots,\,a_m,\,f)(x)\big|\\
&&\quad\lesssim\sum_{l}\int_{\mathbb{R}^{m+1}}\frac{|P_2(A,\,x,\,y_{m+1})|}{(\sum_{k=1}^{m+1}|x-y_k|)^{m+2}}
\phi\Big(\frac{|y_{m+1}-y_1|}{|I_l|}\Big)|b_1^l(y_1)|\\
&&\qquad\qquad\qquad\times\prod_{j=2}^m|a_j(y_j)||f(y_{m+1})|{\rm d}\vec{y}.
\end{eqnarray*}
Observe that
$$\int_{I_l}|b_1^l(y_1)|{\rm d}y_1\lesssim |I_l|,$$
and for $x\in\mathbb{R}\backslash \widetilde{\Omega}$,
\begin{eqnarray*}
&&\int_{\mathbb{R}^{m-1}}\frac{1}{(\sum_{k=1}^{m+1}|x-y_k|)^{m+2}}\prod_{j=2}^m|a_j(y_j)|{\rm d}y_2\dots{\rm d}y_m\\
&&\quad\lesssim\frac{1}{|x-y_{m+1}|^3}\prod_{j=2}^mMa_j(x).
\end{eqnarray*}
Let $${\rm E}(x)=\sum_{l}|I_l|\Big(\int_{4I_l}\frac{|P_2(A;\,x,\,y_{m+1})|}{|x-y_{m+1}|^{3}}|f(y_{m+1})|{\rm d}y_{m+1}\Big).$$
Thus,
$$
\big|\mathcal{C}_{m+1,A}(a_1^3,\,a_2,\dots,\,a_m,\,f)(x)\big|\lesssim\prod_{j=2}^mMa_j(x){\rm E}(x).
$$Set
\begin{eqnarray}\label{eq2.09}A_{I_l}(y)=A(y)-\langle A'\rangle_{I_l}y.\end{eqnarray}
It is easy to verify that for  all $y,\,z\in \mathbb{R}$,
$$P_2(A;\,y,\,z)=P_2(A_{I_l};\,y,\,z).$$
A straightforward computation involving Lemma \ref{lem2.1} shows that for  $y_{m+1}\in 4I_l$,
$$|A_{I_l}(x)-A_{I_l}(y_{m+1})|\lesssim |x-y_{m+1}|\big(1+\big|\langle A'\rangle_{I_l}-\langle A'\rangle_{I_x^{y_{m+1}}}\big|\big).$$
Thus,
\begin{eqnarray*}
\int_{\mathbb{R}\backslash \widetilde{\Omega}}\frac{|P_2(A;\,x,\,y_{m+1})|}{|x-y_{m+1}|^{3}}{\rm d}x&\lesssim&\sum_{k=2}^{\infty}\int_{2^kI_l}(k+|A'(y_{m+1})-\langle A'\rangle_{I_l}|\big)\frac{{\rm d}x}{|x-y_{m+1}|^{2}}\\
&\lesssim&|I_l|^{-1}(1+|A'(y_{m+1})-\langle A'\rangle_{I_l}|).
\end{eqnarray*}
This, via the generalization of H\"older's inequality, yields
\begin{eqnarray*}
\int_{\mathbb{R}\backslash \widetilde{\Omega}}\int_{4I_l}\frac{|P_2(A;\,x,\,y)|}{|x-y|^3}|f(y)|{\rm d}y{\rm d}x&\lesssim&|I_l|^{-1}\int_{4I_l}|f(y)|| A'(y)-\langle A'\rangle_{I_l}|{\rm d}y\\
&\lesssim &\|f\|_{L\log L,\,4I_l}.
\end{eqnarray*}
Combining the estimates above then yields
$$\int_{\mathbb{R}\backslash \widetilde{\Omega}}{\rm E}(x){\rm d}x\lesssim \sum_{l}|I_l|\|f\|_{L\log L,\,4I_l}\lesssim\sum_{l}|I_l|
+\int_{\mathbb{R}}|f(y)|\log ({\rm e}+|f(y)|){\rm d}y,$$
since $$\|f\|_{L\log L,\,4I_l}\lesssim 1+\frac{1}{|4I_l|}\int_{4I_l}|f(y)|\log ({\rm e}+|f(y)|){\rm d}y,$$
see \cite[p. 69]{rr}. Thus,
\begin{eqnarray*}
&&\big|\big\{x\in \mathbb{R}:\,\big|\mathcal{C}_{m+1,A}\big(a_1^{3},\,a_2\dots,\,a_{m},\,f\big)(x)|>1\big\}\big|\\
&&\quad\lesssim |\widetilde{\Omega}|+\sum_{j=2}^{m}|\{x\in\mathbb{R}:\,Ma_j(x)>1\}|
+\big|\big\{x\in\mathbb{R}\backslash \widetilde{\Omega}:\,{\rm E}(x)>1\big\}\big|\\
&&\quad\lesssim\sum_{j=2}^{m}\|a_j\|_{L^{q_j}(\mathbb{R})}^{q_j}+\int_{\mathbb{R}^n\backslash \widetilde{\Omega}}{\rm E}(x){\rm d}x\\
&&\quad\lesssim\|a_1\|_{L^{p_1}(\mathbb{R})}^{p_1}+\sum_{j=2}^{m}\|a_j\|_{L^{q_j}(\mathbb{R})}^{q_j}
+\int_{\mathbb{R}}|f(x)|
\log\big({\rm e}+|f(x)|\big){\rm d}x.
\end{eqnarray*}
This establishes (\ref{eq2.08}) for the case of $q_{m+1}=1$. For the case of $q_{m+1}\in (1,\,\infty)$, it follows from H\"older's inequality that
\begin{eqnarray*}\sum_{l}|I_l|\|f\|_{L\log L,\,4I_l}&\lesssim&\sum_{l}|I_l|^{1-1/q_{m+1}}\Big(\int_{4I_l}|f(y)|^{q_{m+1}}{\rm d}y\Big)^{1/q_{m+1}}\\
&\lesssim&\sum_{l}|I_l|+\sum_{l}\int_{4I_l}|f(y)|^{q_{m+1}}{\rm d}y.
\end{eqnarray*}
Thus, the inequality (\ref{eq2.08}) still holds for $q_{m+1}\in (1,\,\infty)$.

With the estimate (\ref{eq2.07}) in hand, applying the argument above to $a_2$ (fix the exponents $p_1,\,q_3,\,\dots,q_m,\,q_{m+1}$),  we can prove that\begin{eqnarray*}&&|\{x\in\mathbb{R}:\,|\mathcal{C}_{m+1,A}(a_1,\dots,a_m;\,f)|>1\}|\\
&&\quad \lesssim\sum_{j=1}^2\|a_j\|_{L^{p_j}(\mathbb{R})}^{p_j}+\sum_{j=3}^{m}\|a_j\|_{L^{q_j}(\mathbb{R})}^{q_j}+
\int_{\mathbb{R}}|f(x)|^{q_{m+1}}
\log^{\beta_{q_{m+1}}}\big({\rm e}+|f(x)|\big){\rm d}x.\nonumber
\end{eqnarray*}
Repeating this procedure $m$ times then leads to our desired conclusion.
\end{proof}
\begin{lemma}\label{lem2.8}Let $A$ be a function on $\mathbb{R}$ such that $A'\in {\rm BMO}(\mathbb{R})$. Then for $s\in (0,\,1/2)$, \begin{eqnarray}\label{eq2.10}M^{\sharp}_{0,\,s}\big(\mathcal{C}_{m+1,A}(a_1,\dots,a_m;\,f)\big)(x)\lesssim M_{L\log L}f(x)\prod_{j=1}^mMa_j(x),\end{eqnarray}
provided that $a_1,\,\dots,\,a_j\in C^{\infty}_0(\mathbb{R})$.
\end{lemma}
\begin{proof}Without loss of generality, we may assume that $\|A'\|_{{\rm BMO}(\mathbb{R})}=1$. Let $x\in \mathbb{R}$, $I\subset \mathbb{R}$ be an interval  containing $x$. Decompose $f$  as
$$f(y)=f(y)\chi_{64I}(y)+f(y)\chi_{\mathbb{R}\backslash 64I}(y):=f^1(y)+f^2(y),$$
and for $j=1,\,\dots,\,m$, $$a_j(y)=a_j(y)\chi_{64I}(y)+a_j(y)\chi_{\mathbb{R}\backslash 64I}(y):=a_j^1(y)+a_j^2(y).$$
By the estimate (\ref{equa:1.4}), we know $|\mathcal{C}_{m+1,A}(a_1,\dots,a_m,\,f^2)(z)|<\infty$ for a. e. $z\in\mathbb{R}$ and we
can choose some $x_I\in 3I\backslash 2I$ such that $|\mathcal{C}_{m+1,A}(a_1,\dots,a_m,\,f^2)(x_I)|<\infty.$ For $\delta\in (0,\,1)$, write
\begin{eqnarray*}
&&\frac{1}{|I|}\int_{I}\Big|\mathcal{C}_{m+1,A}(a_1,\dots,a_m,\,f)(y)-\mathcal{C}_{m+1,A}(a_1,\dots,a_m,\,f^2)(x_I)\Big|^{\delta}{\rm d}y\\
&&\quad\lesssim\frac{1}{|I|}\int_{I}\big |\mathcal{C}_{m+1,A}(a_1,\dots,a_m;\,f^1)(y)\big|^{\delta}{\rm d}y\\
&&\qquad+\sum_{\Lambda}\frac{1}{|I|}\int_{I}\big |\mathcal{C}_{m+1,A}(a_1^{i_1},\dots,a_m^{i_m};\,f^2)(y)\big|^{\delta}{\rm d}y\\
&&\qquad+\frac{1}{|I|}\int_{I}\Big|\mathcal{C}_{m+1,A}(a_1^2,\dots,a_m^2;f^2)(y)-\mathcal{C}_{m+1,A}(a_1^2,\dots,a_m^2;f^2)(x_I)\Big|^{\delta}{\rm d}y\\
&&\quad:={\rm I}+{\rm II}+{\rm III},
\end{eqnarray*}
where $\Lambda=\{(i_1,\,\dots,\,i_m):\, i_1,\,\dots,\,i_m\in \{1,\,2\},\,\,\min_{j}i_j=1\}$. It follows from Lemma \ref{lem2.4} that
$$
{\rm I}^{1/\delta}\lesssim M_{L\log L}f(x)\prod_{j=1}^{m}Ma_j(x).
$$

We turn our attention to the term ${\rm III}$. Let $A_I$ be defined in (\ref{eq2.09}). Applying Lemma \ref{lem2.1} and the John-Nirenberg inequality, we can verify that if $y\in I$, and $z\in4^{l+1}I\backslash 4^lI$ with $l\in\mathbb{N}$, then
\begin{eqnarray}\label{eq2.11}
|P_2(A_I;y,z)|\lesssim\big(l+|A'(z)-\langle A'\rangle_I|\big)|y-z|.
\end{eqnarray}
This, along with another application of Lemma \ref{lem2.1}, gives us that for $y\in I$ and $z_{m+1}\in4^{l+1}I\backslash 4^lI$,
\begin{eqnarray}\label{eq2.12}
\,\,&&\Big|\frac{P_2(A_I;y,z_{m+1})}{|y-z_{m+1}|}-
\frac{P_2(A_I;x_I,z_{m+1})}{|x_I-z_{m+1}|}\Big|\\
&&\quad\le\frac{|A_I(y)-A_I(x_I)|}{|y-z_{m+1}|}+|P_2(A_I;\,x_I,z_{m+1})|\Big|\frac{1}{|x_I-z_{m+1}|}-\frac{1}{|y-z_{m+1}|}\Big|\nonumber\\
&&\quad\le \big(l+A'(z_{m+1})-\langle  A'\rangle_I|\big)\frac{|y-x_I|}{|x_I-z_{m+1}|}.\nonumber
\end{eqnarray}
We now deduce  from  Lemma \ref{lem2.3} and (\ref{eq2.11})  that
\begin{eqnarray*}
&&\int_{\mathbb{R}^{m+1}}\big|K(y;z_1,\dots,z_{m+1})-K(x_I;z_1,\dots,z_{m+1})\big|\\
&&\qquad\qquad\times\frac{|P_2(A_I;y,z_{m+1})|}{|y-z_{m+1}|}
\prod_{j=1}^m|a^2_j(z_j)||f(z_{m+1})|{\rm d}\vec{z}\\
&&\quad\lesssim\sum_{l=3}^{\infty}l2^{-\gamma l}\prod_{j=1}^{m}\Big(\frac{1}{|4^lI|}\int_{4^l I}|a_j(z_j)|{\rm d}z_j\Big)\\
&&\qquad\times\Big(\frac{1}{|4^lI|}\int_{4^l I}\big| A'(z_{m+1})-\langle A'\rangle_I\big||f(z_{m+1})|{\rm d}z_{m+1}\Big)\\
&&\quad\lesssim  M_{L\log L}f(x)\prod_{j=1}^{m}Ma_j(x).
\end{eqnarray*}

On the other hand, we obtain from (\ref{eq2.12}) and the size condition (\ref{eq2.03}) that
\begin{eqnarray*}
&&\int_{\mathbb{R}^{m+1}}\big|K(x_I;z_1,\dots,z_{m+1})\big|\Big|\frac{P_2(A_I;y,z_{m+1})}{|y-z_{m+1}|}-
\frac{P_2(A_I;x_I,z_{m+1})}{|x_I-z_{m+1}|}\Big|\\
&&\quad\qquad\times\prod_{j=1}^m|a^2_j(z_j)||f^2(z_{m+1})|{\rm d}\vec{z}\lesssim M_{L\log L}f(x)\prod_{j=1}^{m}Ma_j(x).
\end{eqnarray*}
Therefore, for each $y\in I$,
\begin{eqnarray}\label{eq2.13}
&&\big|\mathcal{C}_{m+1,\,A}(a_1,\,\dots,\,a_m;\,f^2)(y)-\mathcal{C}_{m+1,\,A}(a_1,\,\dots,\,a_m;\, f^2)(x_I)\big|\\
&&\quad\lesssim M_{L\log L}f(x)\prod_{j=1}^{m}Ma_j(x),\nonumber
\end{eqnarray}
which shows that
$${\rm III}^{1/\delta}\lesssim M_{L\log L}f(x)\prod_{j=1}^{m}Ma_j(x).$$

It remains to estimate ${\rm II}$. For simplicity, we assume that for some $l_0\in\mathbb{N}$, $i_1=\dots=i_{l_0}=1$ and
$l_{l_0+1}=\dots=i_m=2$. Observe that for $y\in I$,
\begin{eqnarray*}
&&\int_{ \mathbb{R}\backslash 64I}\frac{|P_2(A_I;\,y,z_{m+1})|}{|y-z_{m+1}|^{\frac{m+1}{m+1-l_0}+1}}|f(z_{m+1})|{\rm d}z_{m+1}\\
&&\quad\lesssim\sum_{k=3}^{\infty}\frac{1}{(4^k|I|)^{\frac{l_0+1}{m-l_0+1}}}\int_{4^kI}\big(k+\big|A'(z_{m+1})-\langle A'\rangle_I\big|\big)|f(z_{m+1})|{\rm d}z_{m+1}\\
&&\quad\lesssim |I|^{-\frac{l_0}{m+1-l_0}}M_{L\log L}f(x),
\end{eqnarray*}
and
\begin{eqnarray*}
\int_{ \mathbb{R}\backslash 64I}\frac{1}{|y-z_{j}|^{\frac{m+1}{m+1-l_0}}}|a_j(z_{j})|{\rm d}z_{j}\lesssim|I|^{-\frac{l_0}{m+1-l_0}}Ma_j(x).
\end{eqnarray*}
This, in turn implies that, for each $y\in I$,
\begin{eqnarray}\label{eq2.14}
\,\,\big|\mathcal{C}_{m+1,\,A}(a_1^{i_1},\dots,a_m^{i_m};f^2)(y)\big|&\lesssim&\prod_{j=1}^{l_0}\int_{64I}|a_j^1(z_j)|{\rm d}z_j\\
&&\times\prod_{j=l_0+1}^m\int_{ \mathbb{R}\backslash 64I}\frac{|a_j(z_{j})|}{|y-z_{j}|^{\frac{m+1}{m+1-l_0}}}{\rm d}z_{j}\nonumber\\
&&\times\int_{ \mathbb{R}\backslash 64I}\frac{|P_2(A_I;\,y,z)|}{|y-z|^{\frac{m+1}{m+1-l_0}+1}}|f(z)|{\rm d}z\nonumber\\
&\lesssim&M_{L\log L}f(x)\prod_{j=1}^{m}Ma_j(x).\nonumber
\end{eqnarray}
Therefore,
$${\rm II}^{1/\delta}\lesssim M_{L\log L}f(x)\prod_{j=1}^{m}Ma_j(x).$$
Combining the estimates for I, II and III leads to (\ref{eq2.10}).
\end{proof}

We are now ready to establish the main result in this section.
\begin{theorem}\label{thm2.1}
Let $A$ be a function on $\mathbb{R}$ such that $A'\in {\rm BMO}(\mathbb{R})$. Then \begin{eqnarray}\label{eq2.15}
&&\big|\{x\in\mathbb{R}:\,|\mathcal{C}_{m+1,\,A}(a_1,\,\dots,\,a_m;\, f)(x)|>1\}\big|\\
&&\quad\lesssim \sum_{j=1}^{m}\|a_j\|_{L^1(\mathbb{R})}+\int_{\mathbb{R}}|f(y)|
\log \big({\rm e}+|f(y)|\big){\rm d}y.\nonumber
\end{eqnarray}
\end{theorem}
\begin{proof} Without loss of generality, we may assume that $a_1,\,\dots,\,a_m\in C^{\infty}_0(\mathbb{R})$. At first, let $q_1,\,\dots,q_{m+1},\,q\in (1,\,\infty)$ with $1/q=1/q_1+\dots+1/q_{m+1}$. Recalling that $\mathcal{C}_{m+1,A}$ is bounded from $L^{\infty}(\mathbb{R})\times \dots\times L^{\infty}(\mathbb{R})\times L^q(\mathbb{R})$ to $L^q(\mathbb{R})$ (see \cite{hy}),  we then know that for bounded functions $a_1\,\dots,\,a_m,\,f$ with compact supports, $$\sup_{\lambda>0}\lambda^{q}|\{x\in\mathbb{R}:\,\mathcal{C}_{m+1,A}(a_1,\dots,a_m;\,f)>\lambda\}|\lesssim \|f\|_{L^q(\mathbb{R})}^q \prod_{j=1}^m\|a_j\|_{L^{\infty}(\mathbb{R})}^q<\infty.$$ This, along with  Lemma \ref{lem2.2} and Lemma \ref{lem2.8}, leads to
\begin{eqnarray}\label{eq2.16}\|\mathcal{C}_{m+1,A}(a_1,\dots,a_m;\,f)\|_{L^{q,\,\infty}(\mathbb{R})}\lesssim
\|f\|_{L^{q_{m+1}}(\mathbb{R})}\prod_{j=1}^m\|a_j\|_{L^{q_j}(\mathbb{R})}.\end{eqnarray}

Now let $r_1\in [1,\,q_1),\,\,\dots,\,r_m\in [1,\,q_m)$ and $1/r=1/r_1+\dots+1/r_m+1/q_{m+1}$. Invoking Lemma \ref{lem2.7}, we deduce from (\ref{eq2.16}) that
$$\big|\big\{x\in\mathbb{R}:|\mathcal{C}_{m+1,A}(a_1,\dots,a_m; f)(x)|>1\big\}\big|\lesssim \sum_{j=1}^{m}\|a_j\|_{L^{r_j}(\mathbb{R}^n)}^{r_j}+\|f\|_{L^{q_{m+1}}(\mathbb{R})}^{q_{m+1}}.
$$
This, via homogeneity, shows that $\mathcal{C}_{m+1,\,A}$ is bounded from $L^{r_1}(\mathbb{R})\times\dots \times L^{r_m}(\mathbb{R})\times L^{q_{m+1}}(\mathbb{R})$ to $L^{r,\,\infty}(\mathbb{R})$.

We now prove that for  $p_1,\,\dots,\,p_m\in (1,\,\infty)$, and $p\in (1/(m+1),\,1)$ such that $1/p=1/p_1+\dots+1/p_m+1$,
\begin{eqnarray}\label{eq2.17}&&|\{x\in\mathbb{R}:\,|\mathcal{C}_{m+1,\,A}(a_1,\dots,a_m;\,f)(x)|>1\}|\\
&&\quad\lesssim\sum_{j=1}^m\|a_j\|_{L^{p_j}(\mathbb{R})}^{p_j}
+\int_{\mathbb{R}}|f(x)|\log ({\rm e}+|f(x)|){\rm d}x.\nonumber
\end{eqnarray}To this aim, we choose $q_1,\,\dots q_{m+1}\in (1,\,\infty)$ such that $1/q=1/q_1+\dots+1/q_{m+1}<1$, and
$p_1*\in [1,\,q_1),\,\dots,\,p_m^*\in [1,\,q_m)$, $p^*\in (0,\,1)$ such that $1/p^*=1/p_1^*+\dots+1/p_m^*+1/q_{m+1}$ and $p^*<p$. Recall that $\mathcal{C}_{m+1,\,A}$ is bounded from $L^{p_1^*}(\mathbb{R})\times\dots \times L^{p_m^*}(\mathbb{R})\times L^{q_{m+1}}(\mathbb{R})$ to $L^{p^*,\,\infty}(\mathbb{R})$. Thus, for bounded functions $a_1,\,\dots,\,a_m,\,f$ with compact support,
$$\lambda^{p^*}|\{x\in\mathbb{R}:\,\mathcal{C}_{m+1,A}(a_1,\dots,a_m;\,f)>\lambda\}|\lesssim \prod_{j=1}^m\|a_j\|_{L^{p_j^*}(\mathbb{R})}^p\|f\|_{L^{p_{m+1}^*}(\mathbb{R})}^p.$$
Let $\psi(t)=t^{p}\log^{-1}({\rm e}+t^{-p})$. A trivial computation gives us that
\begin{eqnarray*}&&\sup_{\lambda>0}\psi(\lambda)|\{x\in\mathbb{R}:\,\mathcal{C}_{m+1,A}(a_1,\dots,a_m;\,f)>\lambda\}|\\
&&\quad\lesssim\sup_{0<\lambda<1}\lambda^{p^*}|\{x\in\mathbb{R}:\,\mathcal{C}_{m+1,A}(a_1,\dots,a_m;\,f)>\lambda\}|\\
&&\qquad+\sup_{\lambda\geq 1}\lambda^{2}|\{x\in\mathbb{R}:\,\mathcal{C}_{m+1,A}(a_1,\dots,a_m;\,f)>\lambda\}|\\
&&\quad\lesssim \|f\|_{L^{q_{m+1}}(\mathbb{R})}^p\prod_{j=1}^m\|a_j\|_{L^{p_j^*}(\mathbb{R})}^p+\|f\|_{L^2(\mathbb{R})}^2\prod_{j=1}^m\|a_j\|_{L^{\infty}(\mathbb{R})}^2
<\infty.\end{eqnarray*}
This, via Lemma \ref{lem2.2}, Lemma \ref{lem2.8} and the estimate (\ref{eq2.01}), tells us that
\begin{eqnarray*}&&|\{x\in\mathbb{R}:\,|\mathcal{C}_{m+1,\,A}(a_1,\dots,a_m;\,f)(x)|>1\}|\\
&&\quad\lesssim\sup_{\lambda>0}\psi(\lambda)\sum_{j=1}^m\big|\{x\in\mathbb{R}:Ma_j(x)>\lambda^{\frac{p}{p_j}}\}\big|\\
&&\qquad+\sup_{\lambda>0}\psi(\lambda)|\{x\in\mathbb{R}: M_{L\log L}f(x)>\lambda^{p}\}\big|\\
&&\quad\lesssim\sup_{\lambda>0}\psi(\lambda)\Big(\lambda^{-p}\sum_{j=1}^m\|a_j\|_{L^{p_j}(\mathbb{R})}^{p_j}+\int_{\mathbb{R}}\frac{|f(x)|}{\lambda^p}
\log \Big({\rm e}+\frac{|f(x)|}{\lambda^p}\Big){\rm d}x\Big)\\
&&\quad\lesssim\sum_{j=1}^m\|a_j\|_{L^{p_j}(\mathbb{R})}^{p_j}
+\int_{\mathbb{R}}|f(x)|\log ({\rm e}+|f(x)|){\rm d}x,
\end{eqnarray*}
and then establishes  (\ref{eq2.17}).

Finally, by (\ref{eq2.17}) and  invoking Lemma \ref{lem2.7} $m$ times, we  obtain the estimate (\ref{eq2.15}). This completes the proof of Lemma \ref{thm2.1}.
\end{proof}
\section{Proof of Theorem \ref{t1.3}}
Let $\mathcal{S}$ be a family of cubes and $\eta\in (0,\,1)$. We say that $\mathcal{S}$ is  an $\eta$-sparse family,  if, for each fixed $Q\in \mathcal{S}$, there exists a measurable subset $E_Q\subset Q$, such that $|E_Q|\geq \eta|Q|$ and $\{E_{Q}\}$ are pairwise disjoint. A sparse family is called simply  sparse if $\eta=1/2$. For a fixed cube $Q$, denote by $\mathcal{D}(Q)$ the set of dyadic cubes with respect to $Q$, that is, the cubes from $\mathcal{D}(Q)$ are formed by repeating subdivision of $Q$ and each of descendants into $2^n$ congruent subcubes.

For constants $\beta_1,\,\dots,\,\beta_m\in [0,\,\infty)$, let $\vec{\beta}=(\beta_1,\,\dots,\,\beta_m)$. Associated with  the sparse family $\mathcal{S}$ and  $\vec{\beta}$, we define  sparse operator $\mathcal{A}_{m;\,\mathcal{S},L(\log L)^{\vec{\beta}}}$  by
$$\mathcal{A}_{m;\,\mathcal{S},\,L(\log L)^{\vec{\beta}}}(f_1,\dots,f_m)(x)=\sum_{Q\in\mathcal{S}}\prod_{j=1}^m\|f_j\|_{L(\log L)^{\beta_j},\,Q}\chi_{Q}(x).$$

\begin{lemma}\label{lem3.1}Let $p_1,\,\dots,\,p_m\in (1,\,\infty)$, $p\in (0,\infty)$ such that $1/p=1/p_1+\dots+1/p_m$, and $\vec{w}=(w_1,\,\dots,\,w_m)\in A_{\vec{P}}(\mathbb{R}^{mn})$. Set $\sigma_i=w_i^{-1/(p_i-1)}$. Let $\mathcal{S}$ be a sparse family. Then
for $\beta_1,\,\dots,\,\beta_m\in [0,\,\infty)$,
$$\big\|\mathcal{A}_{m;\mathcal{S},L(\log L)^{\vec{\beta}}}(f_1,\dots,f_m)\big\|_{L^p(\mathbb{R}^n,\nu_{\vec{w}})}\lesssim [\vec{w}]_{A_p}^{\max\{1,\frac{p_1'}{p},\dots,\frac{p_m'}{p}\}}\prod_{j=1}^m[\sigma_j]_{A_{\infty}}^{\beta_j}
\|f_j\|_{L^{p_j}(\mathbb{R}^n,w_j)}.$$
If $\vec{w}\in A_{1,\,\dots,\,1}(\mathbb{R}^{mn})$, then
\begin{eqnarray*}&&\nu_{\vec{w}}(\{x\in\mathbb{R}^n:\,\mathcal{A}_{m;\,\mathcal{S},L(\log L)^{\vec{\beta}}}(f_1,\,\dots,\,f_m)(x)>1\})\\
&&\quad\lesssim\prod_{j=1}^m\Big(\int_{\mathbb{R}^n}
|f_j(y_j)|\log^{|\beta|}\big(1+|f_j(y_j)|\big)w_j(y_j){\rm d}y_j\Big)^{\frac{1}{m}},
\end{eqnarray*}
with $|\beta|=\sum_{j=1}^m|\beta_j|$.
\end{lemma}
For the proof of  Lemma \ref{lem3.1}, see \cite{chenhu}.

In the following, we say that $U$ is an $m$-sublinear operator, if $U$ satisfies that for each $i$ with $1\leq i\leq m$,
\begin{eqnarray*}U(f_1,\dots,f_i^1+f_i^2,f_{i+1},\dots,f_m)(x)&\le &U(f_1,\dots,f_i^1,f_{i+1},\dots,f_m)(x)\\
&&+
U(f_1,\dots,f_i^2,\,f_{i+1},\dots,f_m)(x),\end{eqnarray*}
and for any $t\in \mathbb{C}$,
$$U(f_1,\dots,tf_i^1,f_{i+1},\dots,f_m)(x)=t U(f_1,\dots,f_i^1,f_{i+1},\dots,f_m)(x).$$

For an $m$-sublinear operator $U$ and $\kappa\in\mathbb{N}$, let $\mathcal{M}_U^{\kappa}$ be the corresponding grand maximal operator, defined by
\begin{eqnarray*}&&\mathcal{M}_{U}^{\kappa}(f_1,\dots,f_m)(x)\\
&&\quad=\sup_{Q\ni x}\big\|U(f_1,\,\dots,\,f_m)(\xi)-U(f_1\chi_{Q^{\kappa}},\,\dots,\,f_m\chi_{Q^{\kappa}})(\xi)\big\|_{L^{\infty}(Q)},\nonumber
\end{eqnarray*}
with $Q^{\kappa}=3^{\kappa} Q$.
This operator was   introduced by Lerner \cite{ler2} and plays an important role in the proof of weighted estimates for singular integral
operators, see \cite{ler3,chenhu,li}.

\begin{lemma}\label{lem3.2}Let $m,\,\kappa\in\mathbb{N}$, $U$ be an $m$-sublinear operator and $\mathcal{M}_U^{\kappa}$  the corresponding grand maximal operator. Suppose that $U$ is bounded from $L^{q_1}(\mathbb{R}^n)\times\dots\times L^{q_m}(\mathbb{R}^n)$ to $L^{q,\,\infty}(\mathbb{R}^n)$ for some $q_1,\,\dots,\,q_m\in (1,\,\infty)$ and $q\in (1/m,\,\infty)$ with $1/q=1/q_1+\dots+1/q_m$.
Then for bounded functions $f_1,\,\dots,\,f_m$, cube $Q_0\subset \mathbb{R}^n$, and a. e. $x\in Q_0$,
$$\big|U(f_1\chi_{Q_0^{\kappa}},\,\dots,\,f_m\chi_{Q_0^{\kappa}})(x)|\lesssim \prod_{j=1}^m|f_j(x)|+\mathcal{M}_{U}^{\kappa}(f_1\chi_{Q_0^{\kappa}},\,\dots,\,f_m\chi_{Q_0^{\kappa}})(x).$$
\end{lemma}

For the proof of Lemma \ref{lem3.2}, see \cite{chenhu, li}.

The following theorem is an extension of Theorem 4.2 in \cite{ler2}, and will be useful in the proof of Theorem \ref{t1.3}.
\begin{theorem}\label{thm3.1}
Let $\beta_1,\,\dots,\,\beta_m\in [0,\,\infty)$, $\kappa,\,m\in\mathbb{N}$ and $U$ be an $m$-sublinear operator and $\mathcal{M}_U^{\kappa}$ be the corresponding grand maximal operator.
Suppose that $U$ is bounded from $L^{q_1}(\mathbb{R}^n)\times\dots\times L^{q_m}(\mathbb{R}^n)$ to $L^{q,\infty}(\mathbb{R}^n)$ for some $q_1,\,\dots,\,q_m\in (1,\,\infty)$ and $q\in (1/m,\,\infty)$ with $1/q=1/q_1+\dots+1/q_m$,
\begin{eqnarray*}&&\big|\big\{x\in\mathbb{R}^n:\,\mathcal{M}_{U}^{\kappa}(f_1,\,\dots,f_m)(x)>1\big\}\big|\\
&&\quad\le C_1\sum_{j=1}^m\int_{\mathbb{R}^n}|f_j(y_j)|\log^{\beta_j}\big({\rm e}+|f_j(y_j)|\big){\rm d}y_j.\nonumber
\end{eqnarray*} Then for bounded functions $f_1,\,\dots,f_m\in L^1(\mathbb{R}^n)$ with compact supports, there exists a $\frac{1}{2}\frac{1}{3^{\kappa n}}$-sparse of family $\mathcal{S}$ such that for a. e. $x\in\mathbb{R}^n$,
\begin{eqnarray*}|U(f_1,\,\dots,\,f_m)(x)|\lesssim \sum_{Q\in\mathcal{S}}\prod_{j=1}^m\|f_j\|_{L(\log L)^{\beta_j},\,Q}\chi_{Q}(x).\end{eqnarray*}
\end{theorem}
\begin{proof} We employ the argument used in \cite{ler2}, together with suitable modifications, see also \cite{chenhu,li}.  As in \cite{chenhu,li}, it suffices to prove  that for each cube $Q_0\subset \mathbb{R}^n$,  there exist pairwise disjoint cubes $\{P_j\}\subset \mathcal{D}(Q_0)$, such that $\sum_{j}|P_j|\leq \frac{1}{2}|Q_0|$ and for a. e. $x\in Q_0$,
\begin{eqnarray}\label{eq3.01}
&&|U(f_1\chi_{Q_0^{\kappa}},\dots,\,f_m\chi_{Q_0^{\kappa}})(x)|\chi_{Q_0}(x)\\
&&\quad\le C\prod_{i=1}^m\|f_i\|_{L(\log L)^{\beta_i},Q_0^{\kappa}}+\sum_{j}|U(f_1\chi_{P_j^{\kappa}},\,\dots,f_m\chi_{P_j^{\kappa}})(x)|\chi_{P_j}(x).\nonumber
\end{eqnarray}To prove this, let
$C_2>1$ which will be chosen later and \begin{eqnarray*}E&=&\big\{x\in Q_0:\,|f_1(x)\dots f_m(x)|> C_2 \prod_{i=1}^m\|f_i\|_{L(\log L)^{\beta_i},\,Q_0^{\kappa}}\big\}\\
&&
\cup\big\{x\in Q_0:\,\mathcal{M}_{U}^{\kappa}(f_1\chi_{Q_0^{\kappa}},\,\dots,f_m\chi_{Q_0^{\kappa}})(x)>C_2\prod_{i=1}^m\|f_i\|_{L(\log L)^{\beta_i},\,Q_0^{\kappa}}\big\}.
\end{eqnarray*}
Our assumption implies that
\begin{eqnarray*}
&&\big|\{x\in Q_0:\,\mathcal{M}_{U}^{\kappa}(f_1\chi_{Q_0^{\kappa}},\dots,f_m\chi_{Q_0^{\kappa}})(x)>C_2\prod_{i=1}^m\|f_i\|_{L(\log L)^{\beta_i},\,Q_0^{\kappa}}\big\}\big|\\
&&\quad\le \frac{C_1}{C_2}\sum_{i=1}^m\int_{Q_0^{\kappa}}\frac{|f_i(y_i)|}{\|f_i\|_{L(\log L)^{\beta_i},\,Q_0^{\kappa}}}\log^{\beta_i} \Big({\rm e}+
\frac{|f_i(y_i)|}{\|f_i\|_{L(\log L)^{\beta_i},\,Q^{\kappa}_0}}\Big){\rm d}y_i\\
&&\quad\leq \frac{C_1}{C_2}|Q_0|.
\end{eqnarray*}since
$$\int_{Q_0^{\kappa}}\frac{|f_i(y_i)|}{\|f_i\|_{L(\log L)^{\beta_i},\,Q_0^{\kappa}}}\log^{\beta_i} \Big({\rm e}+
\frac{|f_i(y_i)|}{\|f_i\|_{L(\log L)^{\beta_i},\,Q^{\kappa}_0}}\Big){\rm d}y_i\leq |Q_0^{\kappa}|.$$
If we choose $C_2$ large enough, our assumption then says  that $|E|\leq \frac{1}{2^{n+2}}|Q_0|.$ Now applying the Calder\'on-Zygmund decomposition to $\chi_E$ on $Q_0$ at level $\frac{1}{2^{n+1}}$, we then obtain a family of pairwise disjoint cubes $\{P_j\}$ such that
$$\frac{1}{2^{n+1}}|P_j|\leq |P_j\cap E|\leq \frac{1}{2}|P_j|,$$ and $|E\backslash \cup_jP_j|=0$. It then follows that $\sum_j|P_j|\leq \frac{1}{2}|Q_0|$,  and $P_j\cap E^c\not =\emptyset.$ Therefore,
\begin{eqnarray}\label{eq3.02}
&&\Big\|U(f_1\chi_{Q_0^{\kappa}},\dots,f_m\chi_{Q_0^{\kappa}})(\xi)-
U(f_1\chi_{P^{\kappa}_j},\dots,f_m\chi_{P^{\kappa}_j})(\xi)\Big\|_{L^{\infty}(P_j)}\\
&&\quad\leq C_2\prod_{i=1}^m\|f_i\|_{L(\log L)^{\beta_i},\,Q_0^{\kappa}}.\nonumber
\end{eqnarray}
Note that
\begin{eqnarray}\label{eq3.03}
&&|U(f_1\chi_{Q_0^{\kappa}},\dots,f_m\chi_{Q_0^{\kappa}})(x)|\chi_{Q_0}(x)\\&&
\quad\leq |U(f_1\chi_{Q_0^{\kappa}},\dots,f_m\chi_{Q_0^{\kappa}})(x)|\chi_{Q_0\backslash \cup_jP_j}(x)\nonumber\\
&&\quad+\sum_j\big|U(f_1\chi_{P^{\kappa}_j},\,\dots,\,f_m\chi_{P^{\kappa}_j})(x)\big|\chi_{P_j}(x)\nonumber\\
&&\quad+\sum_j\Big\|U(f_1\chi_{Q_0^{\kappa}},\dots,f_m\chi_{Q_0^{\kappa}})-
U(f_1\chi_{P^{\kappa}_j},\dots,f_m\chi_{P^{\kappa}_j})\Big\|_{L^{\infty}(P_j)}\chi_{P_j}(x).\nonumber\end{eqnarray}
(\ref{eq3.01}) now follows from (\ref{eq3.02}), (\ref{eq3.03}) and Lemma \ref{lem3.2} immediately. This completes the proof of Theorem \ref{thm3.1}.
\end{proof}
For $s\in (0,\,\infty)$, let $M_{s}$ be the maximal operator defined by
$$M_{s}f(x)=\big(M(|f|^{s})(x)\big)^{1/s}.$$
It was proved in \cite[p. 651]{huli} that for $s\in (0,\,1)$ and $\lambda>0$,
\begin{eqnarray}\label{eq3.04}&&
\big|\big\{x\in\mathbb{R}^n:\,M_{s}h(x)>\lambda\big\}\big|\lesssim\lambda^{-1}\sup_{t\geq 2^{-1/{s}}\lambda}t\big|\{x\in\mathbb{R}^n:
|h(x)|>t\}\big|.\end{eqnarray}

{\it Proof of Theorem \ref{t1.3}}. By Lemma \ref{lem3.1}, Theorem \ref{thm3.1} and (\ref{eq2.16}), it suffices to prove that
the grand maximal operator $\mathcal{M}^3_{\mathcal{C}_{m+1,\,A}}$ satisfies that
\begin{eqnarray}\label{eq3.05}&&\big|\{x\in\mathbb{R}:\, \mathcal{M}^3_{\mathcal{C}_{m+1,\,A}}(a_1,\dots,a_m;\,f)(x)>1\}\big|\\
&&\quad\lesssim \sum_{j=1}^{m}\|f_j\|_{L^1(\mathbb{R})}+
\int_{\mathbb{R}}|f(y)|\log\big({\rm e}+|f(y)|\big){\rm d}y.\nonumber
\end{eqnarray}
We assume that $\|A'\|_{{\rm BMO}(\mathbb{R})}=1$ for simplicity.

Let $x\in\mathbb{R}$ and $I$ be a interval containing $x$. For $j=1,\,\dots, m$, set
$$a_j^1(y)=a_j(y)\chi_{27I}(y),\,\,a_j^2(y)=a_j(y)\chi_{\mathbb{R}\backslash 27I}(y).$$
Also, let $$f^1(y)=f(y)\chi_{27I}(y),\,\,f^2(y)=f(y)\chi_{\mathbb{R}\backslash 27I}(y).$$
Set $$\Lambda_1=\{(i_1,\,\dots,i_{m+1}):\,i_1,\dots,i_{m+1}\in\{1,\,2\},\max_{1\leq j\leq m+1}i_j=2,\,\min_{1\leq j\leq m+1}i_j=1\}.$$
Let $A_I(y)$ be the same as in (3.9).
For each fixed $z\in 2I\backslash \frac{3}{2}I$, write
\begin{eqnarray*}
&&\big|\mathcal{C}_{m+1,\,A}(a_1,\,\dots,\,a_m;\,f)(\xi)-\mathcal{C}_{m+1,\,A}(a_1\chi_{27I},\,\dots,\,a_m\chi_{27I};\,f\chi_{27I})(\xi)\big|\\
&&\quad\leq |\mathcal{C}_{m+1,\,A_I}(a_1^2,\dots,a_m^2;f^2)(\xi)-\mathcal{C}_{m+1,\,A_I}(a_1^2,\,\dots,\,a_m^2;\,f^2)(z)|\\
&&\qquad+|\mathcal{C}_{m+1,A_I}(a_1^2,\dots,a_m^2;f^2)(z)|\\
&&\qquad+\sum_{(i_1,\dots,i_m)\in \Lambda_1}\big|\mathcal{C}_{m+1,A_I}(a_1^{i_1},\dots,a_m^{i_m};f^{i_{m+1}})(\xi)\big|\\
&&\quad={\rm D}_1(\xi,\,z)+{\rm D}_2(z)+{\rm D}_3(\xi).
\end{eqnarray*}
As in the estimate (\ref{eq2.13}), we know that for each $z\in 2I\backslash\frac{3}{2}I$,
$${\rm D}_1(\xi,\,z)\lesssim M_{L\log L}f(x)\prod_{j=1}^{m}Ma_j(x).$$

We turn our attention to ${\rm D}_3$. We claim that
for each $y\in 2I$,
\begin{eqnarray}\label{eq3.06}
&&\sum_{(i_1,\,\dots,i_m)\in \Lambda_1}|\mathcal{C}_{m+1;\,A_I}(a_1^{i_1},\dots,a_m^{i_m};\,f^{i_{m+1}})(y)|\lesssim M_{L\log L}f(x)\prod_{j=1}^{m}Ma_j(x).
\end{eqnarray}To see this, we  consider the following two cases.

{\bf Case I}: $i_{m+1}=1$. In this case, $\max_{1\leq k\leq m}i_k=2$. We only consider the case that $i_1=\dots=i_{m-1}=1$ and $i_{m}=2$.
It follows from the size condition (\ref{eq2.06}) that in this case,
\begin{eqnarray*}
|\mathcal{C}_{m+1;\,A_I}(a_1^{i_1},\dots,a_m^{i_m},f^{1})(y)|&\lesssim&\prod_{j=1}^{m-1}\int_{27I}|a_j(y_j)|{\rm d}y_j\int_{\mathbb{R}\backslash 27I}\frac{|a_m(y)|}{|x-y_m|^{m+2}}{\rm d}y\\
&&\qquad\times
\int_{27I}|f(y)||P_2(A_I;y,\,z)|{\rm d}z.
\end{eqnarray*}
Let $q\in (1,\,\infty)$. Another application of Lemma \ref{lem2.1} shows that for $y\in 2I$ and $z\in I$,
\begin{eqnarray*}
\big|A_{I}(z)-A_{I}(y)\big|&\lesssim &|z-y|\Big(\frac{1}{|I_z^{y}|}\int_{I_{z}^{y}}\big|A'(w)-\langle A'\rangle_{I}\big|^qdw\Big)^{1/q}\\
&\lesssim&|z-y|\Big(1+\log\frac{|I|}{|z-y|}\Big)\lesssim |I|,
\end{eqnarray*}
and in this case,
$$|P_2(A_I;z,\,y)|\lesssim |I|(1+|A'-\langle A'\rangle_I|).$$
We thus get
$$|\mathcal{C}_{m+1;\,A_I}(a_1^{i_1},\dots,a_m^{i_m};\,f^{1})(y)|\lesssim M_{L\log L}f(x)\prod_{j=1}^{m}Ma_j(x).
$$

{\bf Case II}:  $i_{m+1}=2$.  As in the   estimates (\ref{eq2.14}), we also have that
$$
\big|\mathcal{C}_{m+1,\,A}(a_1^{i_1},\dots,a_m^{i_m},f^2)(y)\big|\lesssim M_{L\log L}f(x)\prod_{j=1}^{m}Ma_j(x).
$$
Our argument for the above three cases leads to (\ref{eq3.06}).

As to the term ${\rm D}_2$, we have by the inequality (\ref{eq3.06}) that for each $z\in 2I$, 
\begin{eqnarray*}
{\rm D}_2(z)&\leq &|\mathcal{C}_{m+1,\,A}(a_1,\dots,a_m;\,f)(z)|+|\mathcal{C}_{m+1,\,A}(a_1^1,\,\dots,a_m^1;f^1)(z)|\\
&&+\sum_{(i_1,\,\dots,i_m)\in \Lambda_1}\big|\mathcal{C}_{m+1;\,A}(a_1^{i_1},\dots,a_m^{i_m};f^{i_{m+1}})(z)\big|\\
&\lesssim&|\mathcal{C}_{m+1;\,A}(a_1^1,\,\dots,\,a_m^1;\,f^1)(z)|+|\mathcal{C}_{m+1;\,A}(a_1,\,\dots,\,a_m;\,f)(z)|\\
&&+M_{L\log L}f(x)\prod_{j=1}^{m}Ma_j(x).
\end{eqnarray*}

We can now conclude the proof of Theorem \ref{t1.3}. The estimates for ${\rm D}_1$, ${\rm D}_2$ and ${\rm D}_3$, via Lemma \ref{lem2.4}, tell us that for any $\tau\in (0,\,\frac{1}{m+2})$,
\begin{eqnarray*}
&&\sup_{\xi\in Q}\Big|\mathcal{C}_{m+1,\,A}(a_1,\,\dots,\,a_m;\,f)(\xi)-\mathcal{C}_{m+1,\,A}(a_1^2,\,\dots,\,a_m^2;f^2)(\xi)\Big|\\
&&\quad\lesssim  M_{L\log L}f(x)\prod_{j=1}^{m}Ma_j(x)\\
&&\qquad+\Big(\frac{1}{|2I|}\int_{2I}\big|\mathcal{C}_{m+1,\,A}(a_1,\,\dots,\,a_m;\,f)(z)\big|^{\tau}{\rm d}z\Big)^{\frac{1}{\tau}}\nonumber\\
&&\qquad+\Big(\frac{1}{|2I|}\int_{2I}|\mathcal{C}_{m+1;\,A}(a_1^1,\,\dots,\,a_m^1;\,f^1)(z)|^{\tau}{\rm d}z\Big)^{\frac{1}{\tau}}\nonumber\\
&&\quad\lesssim M_{L\log L}f(x)\prod_{j=1}^{m}Ma_j(x)+M_{\tau}\big(\mathcal{C}_{m+1,\,A}(a_1,\,\dots,\,a_m;\,f)\big)(x),\nonumber
\end{eqnarray*}
which implies that
\begin{eqnarray}\label{eq3.07}
\mathcal{M}^3_{\mathcal{C}_{m+1,\,A}}(a_1,\,\dots,\,a_m;\,f)(x)&\lesssim &M_{L\log L}f(x)\prod_{j=1}^{m}Ma_j(x)\\
&&+M_{\tau}\big(\mathcal{C}_{m+1;A}(a_1,\dots,a_m;\,f)\big)(x).\nonumber
\end{eqnarray}
Applying the inequality (\ref{eq3.04}) and Theorem \ref{thm2.1}, we obtain  that
\begin{eqnarray}\label{eq3.08}
&&\big|\big\{x\in\mathbb{R}:\,M_{\tau}\big(\mathcal{C}_{m+1,\,A}(a_1,\,\dots,\,a_m;\,f)\big)(x)>1\}\big|\\
&&\quad\lesssim\sup_{s\geq 2^{-\frac{1}{(m+1)\tau}}} s^{\frac{1}{m+1}}\big|\{x\in\mathbb{R}:|\mathcal{C}_{m+1,\,A}(a_1,\,\dots,\,a_m;\,f)|>s\}\big|\nonumber\\
&&\quad\lesssim\sum_{j=1}^{m}\|a_j\|_{L^1(\mathbb{R})}+\int_{\mathbb{R}}|f(y)|\log ({\rm e}+|f(y)|){\rm d}y.\nonumber
\end{eqnarray}
Combining the estimates (\ref{eq3.07}) and (\ref{eq3.08}) shows that $\mathcal{M}^3_{\mathcal{C}_{m+1,\,A}}$ satisfies (\ref{eq3.05}). This completes the proof of Theorem \ref{t1.3}.
\qed

\end{document}